\documentclass[12pt,reqno]{amsart}
\usepackage[a4paper]{geometry}

\usepackage{color}

\usepackage{amssymb,amsmath,amsthm,enumerate,verbatim,bbm}
\usepackage{mathrsfs}

\DeclareMathOperator{\Tr}{Tr}

\DeclareMathOperator{\meas}{meas}

\DeclareMathOperator*{\Res}{Res}

\renewcommand{\Re}{\operatorname{Re}}

\newcommand{\abs}[1]{\lvert#1\rvert}
\newcommand{\Abs}[1]{\left\lvert#1\right\rvert}
\newcommand{\norm}[1]{\lVert#1\rVert}

\newcommand{\Norm}[1]{\left\lVert#1\right\rVert}
\newcommand{\jap}[1]{\langle#1\rangle}

\newcommand{\bbR}{{\mathbb R}}
\newcommand{\bbC}{{\mathbb C}}

\newcommand{\bbN}{{\mathbb N}}
\newcommand{\bbZ}{{\mathbb Z}}

\newcommand{\bfe}{\mathbf{e}}

\newcommand{\calL}{\mathcal{L}}

\newcommand{\calP}{\mathcal{P}}

\newcommand{\calO}{\mathcal{O}}

\newcommand{\bS}{\mathbf{S}}

\numberwithin{equation}{section}

\theoremstyle{plain}
\newtheorem{theorem}{\bf Theorem}[section]
\newtheorem*{theorem*}{Theorem 1.1$'$}
\newtheorem{lemma}[theorem]{\bf Lemma}

\newtheorem{corollary}[theorem]{\bf Corollary}

\theoremstyle{definition}

\theoremstyle{remark}
\newtheorem*{remark*}{\bf Remark}

\newcommand{\eps}{\varepsilon}

\newcommand{\const}{\mathrm{const}}

\newcommand{\bT}{{\mathbf{T}}}

\renewcommand{\[}{\begin{equation}}
\renewcommand{\]}{\end{equation}}

\begin{document}

\title[Spectral asymptotics for LCM matrices]{Spectral asymptotics for a family of LCM matrices}

\author{Titus Hilberdink}
\address{Department of Mathematics, University of Reading, Whiteknights, PO Box 220, Reading, RG6 6AX, U.K.}
\email{t.w.hilberdink@reading.ac.uk}

\author{Alexander Pushnitski}
\address{Department of Mathematics, King's College London, Strand, London, WC2R~2LS, U.K.}
\email{alexander.pushnitski@kcl.ac.uk}

\date{October 2021}

\dedicatory{To Nikolai Nikolski on the occasion of his 80th birthday with warmest wishes}



\begin{abstract}
We consider the family of arithmetical matrices given explicitly by 
$$
E(\sigma,\tau)=
\left\{\frac{n^\sigma m^\sigma}{[n,m]^\tau}\right\}_{n,m=1}^\infty,
$$
where $[n,m]$ is the least common multiple of $n$ and $m$ and the real parameters $\sigma$ and $\tau$ satisfy $\rho:=\tau-2\sigma>0$, $\tau-\sigma>\frac12$ and $\tau>0$. We prove that $E(\sigma,\tau)$ is a compact self-adjoint positive definite operator on $\ell^2(\bbN)$, and the ordered sequence of eigenvalues of $E(\sigma,\tau)$ obeys the asymptotic relation
$$
\lambda_n(E(\sigma,\tau))=\frac{\varkappa(\sigma,\tau)}{n^\rho}+o(n^{-\rho}), \quad n\to\infty,
$$
with some $\varkappa(\sigma,\tau)>0$. 
We give an application of this fact to the asymptotics of singular values of truncated multiplicative Toeplitz matrices with the symbol given by the Riemann zeta function on the vertical line with abscissa $\sigma<1/2$. We also point out a connection of the spectral analysis of $E(\sigma,\tau)$ to the theory of generalised prime systems. 
\end{abstract}

\maketitle

\section{Introduction}\label{sec.zz}

\subsection{Main result}
For real numbers $\tau$ and $\sigma$, let us consider the infinite matrix
\begin{equation}
E(\sigma,\tau)=
\left\{\frac{n^\sigma m^\sigma}{[n,m]^\tau}\right\}_{n,m=1}^\infty,
\label{E}
\end{equation}
where $[n,m]$ is the least common multiple of $n$ and $m$. We consider $E(\sigma,\tau)$ as a (potentially unbounded) linear operator on the Hilbert space $\ell^2(\bbN)$. We denote 
\begin{equation}
\rho:=\tau-2\sigma. 
\label{rho}
\end{equation}
Observe that $-\rho$ is the degree of homogeneity of the entries of $E(\sigma,\tau)$, i.e. 
$$
\frac{(kn)^\sigma (km)^\sigma}{[kn,km]^\tau}=k^{-\rho}\frac{n^\sigma m^\sigma}{[n,m]^\tau}
$$
for any $k\in\bbN$.

For $q>0$, we denote by $\bS_q$ the standard Schatten class of compact operators $A$, defined by the condition $\sum_{k=1}^\infty s_k(A)^q<\infty$, where $\{s_k(A)\}_{k=1}^\infty$ are the singular values of $A$. In particular, $\bS_2$ is the Hilbert-Schmidt class.

Our main result is:

\begin{theorem}\label{thm1}
Assume that the exponents $\sigma$ and $\tau$ satisfy 
\begin{equation}
\rho>0, \quad \tau+\rho>1, \quad \tau>0.
\label{a0}
\end{equation}
Then the operator $E(\sigma,\tau)$, defined on $\ell^2(\bbN)$ by \eqref{E}, is compact, positive definite and has a trivial kernel. Let $\{\lambda_n(E(\sigma,\tau))\}_{n=1}^\infty$ be the sequence of eigenvalues of $E(\sigma,\tau)$, enumerated in non-increasing order with multiplicities taken into account. Then for some positive constant $\varkappa(\sigma,\tau)$ (given by \eqref{b19}, \eqref{b16} below) we have
\begin{equation}
\lambda_n(E(\sigma,\tau))=\frac{\varkappa(\sigma,\tau)}{n^{\rho}}+{o}(n^{-\rho}), \quad n\to\infty.
\label{s9}
\end{equation}
In particular, $E(\sigma,\tau)\in\bS_q$ if and only if $q\rho>1$. 
\end{theorem}
The proof is given in Sections~\ref{sec.a} and \ref{sec.d}. In Section~\ref{sec.d4} we also discuss a minor improvement to the error  bound in the asymptotic formula \eqref{s9}.

In some particular cases it is possible to compute the constant $\varkappa(\sigma,\tau)$ explicitly. For example, if $\rho=1$ (i.e. $\tau=1+2\sigma$), we have $\varkappa(\sigma,1+2\sigma)=1$. Furthermore, if $\rho=\frac12$ (i.e. $\tau=\frac12+2\sigma$), we have 
\begin{equation}
\varkappa(\sigma,\tfrac12+2\sigma)=\sqrt{\zeta(2+4\sigma)}/\zeta(1+2\sigma);
\label{kappa}
\end{equation}
we give these computations in Section~\ref{sec.d3}. 

To explain the significance of \eqref{kappa}, we first observe that for all $\sigma$ and $\tau$, the diagonal elements of the matrix $E(\sigma,\tau)$ are $\{n^{-\rho}\}_{n=1}^\infty$. Comparing this with \eqref{s9}, one may question whether $\varkappa(\sigma,\tau)=1$ in all cases, i.e. whether only the diagonal part of $E(\sigma,\tau)$ is ``responsible'' for the eigenvalue asymptotics.  Formula \eqref{kappa}  shows that this is not the case.

We would like to comment on the sharpness of the hypotheses \eqref{a0} of Theorem~\ref{thm1}.
Since the diagonal elements of $E(\sigma,\tau)$ are $\{n^{-\rho}\}_{n=1}^\infty$, condition $\rho>0$ is necessary for the compactness of $E(\sigma,\tau)$. Applying $E(\sigma,\tau)$ to the first element of the standard basis in $\ell^2(\bbN)$, we see that condition $\tau+\rho>1$ is necessary for the boundedness of $E(\sigma,\tau)$. In fact, $E(\sigma,\tau)$ is bounded under the assumptions $\rho>0$ and $\tau+\rho>1$ only, see Theorem~\ref{thm.c3}.

The role of the condition $\tau>0$ is different: it is responsible for the positive definiteness of $E(\sigma,\tau)$. In fact, the determinant of the top left $2\times2$ submatrix of $E(\sigma,\tau)$ equals $2^{2\sigma}(1-2^{-2\tau})$, which already shows that positive definiteness breaks down when $\tau<0$.

\subsection{Alternative forms of $E(\sigma,\tau)$ and related literature}
Denote by $(n,m)$ the greatest common divisor of $n$ and $m$. Observing that 
$$
(n,m)[n,m]=nm,
$$
we see that our matrix $E(\sigma,\tau)$ can be alternatively written as
$$
E(\sigma,\tau)=
\left\{
\frac{(n,m)^\tau}{n^{\tau-\sigma}m^{\tau-\sigma}}
\right\}_{n,m=1}^\infty 
$$
or, separating explicitly the homogeneous part, as 
$$
E(\sigma,\tau)=
\left\{
\frac1{n^{\rho/2}m^{\rho/2}}\left(\frac{(n,m)}{[n,m]}\right)^{\tau/2}
\right\}_{n,m=1}^\infty .
$$
The study of arithmetical matrices involving $(n,m)$ or $[n,m]$ goes back to the paper \cite{Smith} by Smith in 1875 who computed the determinant of the $N\times N$ matrix  $\{(n,m)\}_{n,m=1}^N$. 
This matrix has become known as the GCD matrix and $\{[n,m]\}_{n,m=1}^N$ as the LCM matrix. 
GCD and LCM matrices and their submatrices (corresponding to a subset of indices with some structure) and generalisations have been studied by many authors from the algebraic point of view: positive definiteness, invertibility, divisibility, determinants, factorisations  and other structural theorems; see e.g. \cite{BL} and references therein. From the algebraic point of view replacing $(n,m)$ by $n^\sigma (n,m) m^\sigma$, say, amounts to a trivial operation of multiplication by two diagonal matrices, and so this literature is mostly concerned with the case $\sigma=0$.  However, replacing $(n,m)$ by some power $(n,m)^r$, $r\in\bbR$, is a non-trivial step, and matrices $\{(n,m)^r\}$ and $\{[n,m]^r\}$ became known as power GCD and power LCM matrices. In particular, the positive definiteness of $E(\sigma,\tau)$ follows from \cite[Theorem 1]{BL}, which gives positive definiteness of a power GCD matrix $\{(n,m)^r\}_{n,m=1}^N$ for any $r>0$.

Since late 1990s, there have also been some interest in studying the $N\to\infty$ asymptotics of eigenvalues of the $N\times N$ truncations of matrices that are variants of $E(\sigma,\tau)$; see \cite{LS,HL,Hauk,HEL,MH}. Of particular relevance to this work is \cite{LS}, where sharp bounds on the smallest and largest eigenvalues of the $N\times N$ truncations of $E(\sigma,\tau)$ have been obtained in the case $\rho=0$ and $\sigma>\frac12$. (The case $\rho=0$ and $\sigma>1$ was covered earlier by Wintner in his insightful paper  \cite{W} of 1944.) The case $\rho=0$ is of interest because $E(\sigma,2\sigma)$ factorises through a multiplicative Toeplitz matrix, see \eqref{tstar} below.

To our knowledge, the spectral properties of $E(\sigma,\tau)$ as an operator on $\ell^2(\bbN)$ (compactness, eigenvalue asymptotics) have not been discussed in the literature, except for the previous work \cite{H2,H} of one of the present authors, where it was proved that $E(\sigma,1)$ is bounded for $\sigma<\frac12$ and is in the Hilbert-Schmidt class $\bS_2$ if $\sigma<\frac14$. 
This agrees with Theorem~\ref{thm1}. 

\subsection{Key ideas of the proof}
For a natural number $n$, let us write its factorisation as a product of powers of primes
$$
n=\prod_{p\text{ prime}}p^{k_p},
$$
where $k=(k_2,k_3,k_5,k_7,\dots)$ is a multi-index, whose components are labeled by prime numbers, each component $k_p$ is a non-negative integer and $k_p=0$ except for finitely many primes $p$. Writing similarly $m=\prod_p p^{j_p}$, we see that the entries of the matrix $E(\sigma,\tau)$ can be written as 
\begin{equation}
\frac{n^\sigma m^\sigma}{[n,m]^\tau}
=
\prod_{p\text{ prime}} \frac{p^{\sigma(j_p+k_p)}}{p^{\tau(j_p\vee k_p)}},
\label{b300}
\end{equation}
where $a\vee b=\max\{a,b\}$.
In other words, the matrix $E(\sigma,\tau)$ can be viewed, at least formally, as an infinite tensor product 
\begin{equation}
E(\sigma,\tau)=\bigotimes_{p\text{ prime}}E_{p}(\sigma,\tau), \quad
E_{p}(\sigma,\tau)=\{p^{\sigma j}p^{-\tau(j\vee k)}p^{\sigma k}\}_{j,k\in\bbZ_+},
\label{b3a}
\end{equation}
where $\bbZ_+=\{0,1,2,\dots\}$. 

Our first step is the spectral analysis of the matrix $E_p(\sigma,\tau)$. It turns out that this matrix is compact for every $p$ and all of its eigenvalues are positive. Let  us denote the non-increasingly ordered sequence of eigenvalues of $E_{p}(\sigma,\tau)$ by $\{\lambda_j(E_{p}(\sigma,\tau))\}_{j=0}^\infty$. We are able to give some useful estimates for this sequence; in particular, in Lemma~\ref{lma.c2} we prove that the top eigenvalue of $E_{p}(\sigma,\tau)$ satisfies
\begin{equation}
\lambda_0(E_{p}(\sigma,\tau))=1+\calO(p^{-(\tau+\rho)}), \quad p\to\infty.
\label{a1}
\end{equation}
Furthermore, we prove that, in accordance with \eqref{b3a}, the eigenvalues of $E(\sigma,\tau)$ are given by the products of the eigenvalues of $E_{p}(\sigma,\tau)$. More precisely, we prove that \emph{there exists an enumeration} of the eigenvalues $\{\lambda_n(E(\sigma,\tau))\}_{n=1}^\infty$ of $E(\sigma,\tau)$ (here we don't insist on the non-increasing order!) such that 
$$
\lambda_n(E(\sigma,\tau))=\prod_{p \text{ prime}} \lambda_{k_p}(E_p(\sigma,\tau)), \quad\text{ if }\quad n=\prod_{p \text{ prime}}p^{k_p}.
$$
Since $k_p=0$ for all but finitely many primes,  the convergence of the above infinite product reduces to the convergence of the infinite product
$$
\prod_{p \text{ prime}} \lambda_{0}(E_p(\sigma,\tau));
$$
the latter product converges by \eqref{a1} because $\tau+\rho>1$ (this is one of our hypotheses \eqref{a0}).

The above product formula  for the eigenvalues of $E(\sigma,\tau)$ is our main tool; in combination with the estimates for the eigenvalues of $E_p(\sigma,\tau)$ it allows us to prove Theorem~\ref{thm1}. 

\subsection{The structure of the paper}
In Section~\ref{sec.b} we consider an application of Theorem~\ref{thm1} to the asymptotics of singular values of truncated multiplicative Toeplitz matrices; this was one of the original motivations for this work. In Section~\ref{sec.a} we prove  estimates for the eigenvalues of  $E_p(\sigma,\tau)$. In Section~\ref{sec.d} we put them together and give a proof of Theorem~\ref{thm1}. In Section~\ref{sec.d4} we also comment on the connection of this problem with generalised prime systems.

\subsection{Dedication}
It is a pleasure to dedicate this paper to Nikolai Nikolski,
whose recent work on the interface of analysis and analytic number theory (see e.g. \cite{Nikolski3}) and general enthusiasm for this circle of problems has served as an inspiration to us. 

\subsection{Acknowledgements}
The authors are grateful to Lyonell Boulton for stimulating discussions. 
The second author was supported by the Ministry of Science and Higher Education of the Russian Federation, contract No. 075-15-2019-1619.

\section{Application to truncated multiplicative Toeplitz matrices}\label{sec.b}

\subsection{Toeplitz and multiplicative Toeplitz matrices}
Let $\varphi$ be a bounded analytic function in the unit disk of the complex plane, defined through its Taylor series
$$
\varphi(z)=\sum_{j=0}^\infty\varphi_j z^j. 
$$
Setting $\varphi_j=0$ for $j<0$ for notational convenience, 
one can associate with $\varphi$ the infinite \emph{Toeplitz matrix}
$$
T(\varphi)=\{\varphi_{j-k}\}_{j,k=0}^\infty,
$$
considered as a linear operator on the Hilbert space $\ell^2(\mathbb Z_+)$, $\bbZ_+=\{0,1,2,\dots\}$. 
The question of the value distribution of $\varphi$ (which is called a \emph{symbol} in this context) is intimately related to the spectral properties of $T(\varphi)$ as well as to those of its finite $N\times N$ truncations
$$
T_N(\varphi)=\{\varphi_{j-k}\}_{j,k=0}^{N-1}.
$$
For example, the operator norm of $T(\varphi)$ coincides with $\sup_{\abs{z}<1}\abs{\varphi(z)}$, whereas the Szeg\H{o} limit theorem relates the distribution of singular values of $T_N(\varphi)$ as $N\to\infty$ with the value distribution of $\abs{\varphi}$ on the unit circle. We refer to the monograph \cite{Nikolski-Toeplitz} for a detailed treatment of this and related subjects.


In a similar way, a function (=\emph{symbol}) given by the Dirichlet series
$$
\psi(s)=\sum_{n=1}^\infty \psi_n n^{-s}, \quad \Re s>0,
$$
generates a \emph{multiplicative Toeplitz matrix} $\bT(\psi)$, considered as a linear operator on the Hilbert space $\ell^2(\bbN)$. For convenience of notation, let us set $\psi_q=0$ if $q$ is a positive rational number which is not an integer; then $\bT(\psi)$ 
is defined by 
$$
\bT(\psi)=\{\psi_{n/m}\}_{n,m=1}^\infty, 
$$
and one can also consider its $N\times N$ truncations 
$$
\bT_N(\psi)=\{\psi_{n/m}\}_{n,m=1}^N.
$$

Again, the spectral properties of $\bT(\psi)$ are related to the value distribution of $\psi$. In full analogy with the additive case, if $\psi(s)$ is bounded in the half-plane $\Re s>0$, then the operator norm of $\bT(\psi)$ is given by 
$$
\norm{\bT(\psi)}=\sup_{\Re s>0}\abs{\psi(s)}.
$$
This fact goes back to Toeplitz's paper \cite{Toeplitz} and was revisited and put in the framework of modern functional analysis in \cite{HLS}. 
The analogue of Szeg\H{o}'s limit theorem is more subtle and depends on the type of truncation; we refer to \cite{Nikolski-Pushnitski} (which is based on \cite{Bedos}) for a detailed discussion. 

\subsection{The matrix $\bT_N(\psi_\sigma)$}
Let us take $\sigma>1$ and consider the symbol 
$$
\psi_\sigma(s)=\zeta(\sigma+s),
$$
where $\zeta$ is the Riemann zeta-function. In this case, the coefficients of the Dirichlet series expansion of $\psi_\sigma$ are $\psi_{\sigma,n}=n^{-\sigma}$ and so the matrix $\bT_N(\psi_\sigma)$ has an explicit structure:
$$
\bT_N(\psi_\sigma)=
\begin{pmatrix}
1&0&0&0&\cdots&0\\
2^{-\sigma}&1&0&0&\cdots&0\\
3^{-\sigma}&0&1&0&\cdots&0\\
4^{-\sigma}&2^{-\sigma}&0&1&\cdots&0\\
\vdots&\vdots&\vdots&\vdots&\ddots&\vdots\\
N^{-\sigma}& \cdot&\cdot&\cdot&\cdots&1
\end{pmatrix}\ .
$$
Moreover, we have a connection to the arithmetical matrix $E(\sigma,\tau)$ (with $\rho=0$) through the identity
\begin{equation}
\bT(\psi_\sigma)^*\bT(\psi_\sigma)=\zeta(2\sigma)E(\sigma,2\sigma), \quad \sigma>1.
\label{tstar}
\end{equation}
For $\sigma\leq1$, the symbol is unbounded in the right half-plane and so the operator $\bT(\psi_\sigma)$ is unbounded as well (in fact, it is no longer even clear how to associate a multiplicative Toeplitz operator to the symbol $\psi_\sigma$, because the definition of $\psi_\sigma$ involves analytic continuation). However, one can still define the matrix $\bT_N(\psi_\sigma)$ as displayed above. 
In recent years, the study of the spectral properties (especially of the operator norm) of $\bT_N(\psi_\sigma)$ has attracted attention in relation to the value distribution of $\zeta$ on the vertical line with abscissa $\sigma$. 
In particular, in \cite{H3} it was proven that for $1/2<\sigma\leq 1$ the norms of  
$\bT_N(\psi_\sigma)$ give a lower bound for the zeta function as follows:
$$
\max_{\abs{t}\leq T}\abs{\zeta(\sigma+it)}\geq \norm{\bT_N(\psi_\sigma)}+o(1),
\quad 
N=T^{\frac23(\sigma-\frac12)-\varepsilon}\to\infty,
$$
where $\varepsilon>0$ can be taken arbitrarily small. In \cite{A,BS} this method was further improved to get sharper (possibly optimal) lower bounds on $\zeta$ on vertical lines including the critical line $\sigma=1/2$. 

\subsection{Convergence of rescaled matrices $\bT_N^*\bT_N$}
Here we are interested in the case $\sigma<\tfrac12$. It was observed in \cite{H2} that we have the entrywise convergence of rescaled matrices
\begin{equation}
\rho N^{-\rho}\bT_N(\psi_\sigma)^*\bT_N(\psi_\sigma)\to E(\sigma,1),\quad N\to\infty, \quad \sigma<\tfrac12
\label{s5}
\end{equation}
where $E(\sigma,1)$ is as in Section~\ref{sec.zz} and $\rho=1-2\sigma$, in agreement with our previous notation \eqref{rho} with $\tau=1$.

It was established in \cite{H2} that for $\sigma<\tfrac14$ the matrix $E(\sigma,1)$ is Hilbert-Schmidt and the rescaled matrices \eqref{s5} converge to $E(\sigma,1)$ in the Hilbert-Schmidt norm. Then in \cite{H}, exploiting the tensor product structure of $E(\sigma,1)$, it was shown that $E(\sigma,1)$ is bounded for $\sigma<\frac{1}{2}$.

Theorem~\ref{thm1} of this paper gives a sharper result for the Schatten class inclusion of $E(\sigma,1)$.
Below we also improve on the convergence result of \cite{H2} by showing that the rescaled matrices \eqref{s5} converge to $E(\sigma,\tau)$ in the Schatten norm $\bS_q$ whenever $q\rho>1$ and $q$ is an even integer. 
In the next theorem we consider the $N\times N$ matrices in the left hand side of \eqref{s5} as operators acting on $\ell^2(\bbN)$. 

\begin{theorem}\label{thm2}
Let $\sigma<\frac12$;
for any \underline{even natural} number $q$ such that $q\rho>1$ we have the Schatten norm convergence
$$
\Norm{\rho N^{-\rho}\bT_N(\psi_\sigma)^*\bT_N(\psi_\sigma)-E(\sigma,1)}_{\bS_q}\to0
$$
as $N\to\infty$. 
\end{theorem}
We give the proof at the end of this section.
The restriction that $q$ is an even integer seems to be of a technical nature and it would be natural to conjecture that it can be lifted. 

Let $\{s_n(\bT_N(\psi_\sigma))\}_{n=1}^N$ be the singular values of $\bT_N(\psi_\sigma)$ and $\{\lambda_n(E(\sigma,1))\}_{n=1}^\infty$ be the eigenvalues of $E(\sigma,1)$, both sequences listed in non-increasing order. 
\begin{corollary}\label{cr3}
Let $\sigma<\frac12$; then we have 
$$
s_n^2(\bT_N(\psi_\sigma))= \frac{1}{\rho}N^\rho\bigl(\lambda_n(E(\sigma,1))+o(1)\bigr), \quad N\to\infty
$$
uniformly over $1\leq n\leq N$.
\end{corollary}

\begin{proof}
For a given $\sigma<\frac12$, one can always find sufficiently large even integer $q$ such that $q\rho>1$. Now apply Theorem~\ref{thm2} and use the fact that any Schatten norm dominates the operator norm; we obtain
$$
\sup_{1\leq n\leq N}\abs{\rho N^{-\rho}s_n^2(\bT_N(\psi_\sigma))-\lambda_n(E(\sigma,1))}\to0
$$
as $N\to\infty$. This is equivalent to the required statement. 
\end{proof}

\begin{remark*}
Corollary~\ref{cr3} extends the same result for $\sigma<\frac14$ in \cite{H2}. 
\end{remark*}

We observe that there is some similarity between Theorem~\ref{thm2} and the circle of problems considered recently in \cite{Bo,BoVi,BoBoGr}. In these papers, the authors consider truncated (classical) Toeplitz matrices corresponding to unbounded symbols with a specific power singularity on the unit circle. They prove that after a suitable rescaling, the singular values of these truncated Toeplitz matrices converge to the singular values of an explicit compact operator. Of course, in our case the singularity of the symbol $\psi_\sigma(it)=\zeta(\sigma+it)$ (as $t\to\infty$) is far more subtle. 

\subsection{Discussion}
We come back to the  comparison of the value distribution of the symbol $\psi_\sigma(it)=\zeta(\sigma+it)$ with the singular value distribution of $\bT_N(\psi_\sigma)$. Although to our knowledge there are no general theorems relating these distributions, it is still of interest to compare them, and the combination of Corollary~\ref{cr3} and Theorem~\ref{thm1} allows us to do so. 

We first recall that by the reflection formula for the Riemann zeta, we have for any $\sigma<1/2$, 
$$
\abs{\zeta(\sigma+it)}=\left(\frac{t}{2\pi}\right)^{\frac12-\sigma}\abs{\zeta(1-\sigma-it)}(1+o(1)), \quad \abs{t}\to\infty.
$$
Let us focus on the simple case $\sigma<0$. By the reflection formula, we have
$$
\sup_{\abs{t}\leq N}\abs{\zeta(\sigma+it)}\asymp N^{\frac12-\sigma}, \quad N\to\infty. 
$$
By Theorem~\ref{thm2}, we also have
$$
\norm{\bT_N(\psi_\sigma)}\asymp N^{\frac12-\sigma}, \quad N\to\infty,
$$
and so the rate of growth of the norm agrees with the supremum of the symbol on the interval $[-N,N]$. However, the value distribution of the symbol on the same interval does NOT agree with the singular value distribution of $\bT_N(\psi_\sigma)$. Indeed, for any sufficiently small $\lambda>0$, we have
$$
\meas\{t\in[-N,N]: \abs{\zeta(\sigma+it)}>\lambda N^{\frac12-\sigma}\}\asymp N,\quad N\to\infty.
$$
On the other hand, by Corollary~\ref{cr3} and the fact that $\lambda_n(E(\sigma,1))\to0$ as $n\to\infty$ we have
$$
\#\{n: s_n(\bT_N(\psi_\sigma))>\lambda N^{\frac12-\sigma}\}=O(1), \quad N\to\infty
$$
for any $\lambda>0$.

\subsection{Proof of Theorem~\ref{thm2}}
For $x>0$ let 
$$
F(x):=\sum_{n\leq x}n^{-2\sigma}=\frac1\rho x^\rho+\calO(1) +\calO(x^{\rho-1}), \quad x\to\infty.
$$
Let us reproduce the calculation of the entries of $\bT_N(\psi_\sigma)^*\bT_N(\psi_\sigma)$ from \cite{H2}. We have
\begin{align*}
[\bT_N(\psi_\sigma)^*\bT_N(\psi_\sigma)]_{nm}
&=
\sum_{r=1}^N[\bT_N(\psi_\sigma)^*]_{nr}[\bT_N(\psi_\sigma)]_{rm}
\\
&=
\sum_{\genfrac{}{}{0pt}{1}{r=1}{n|r \text{ and } m|r}}^N (r/n)^{-\sigma}(r/m)^{-\sigma}
=
n^\sigma m^\sigma \sum_{\genfrac{}{}{0pt}{1}{r=1}{n|r \text{ and } m|r}}^N r^{-2\sigma}.
\end{align*}
Setting $r=k[n,m]$, we obtain 
\begin{align}
[\bT_N(\psi_\sigma)^*\bT_N(\psi_\sigma)]_{nm}
&=
n^\sigma m^\sigma
\sum_{k\leq N/[n,m]} (k[n,m])^{-2\sigma}
\notag
\\
&=
\frac{n^\sigma m^\sigma}{[n,m]^{2\sigma}}\sum_{k\leq N/[n,m]}k^{-2\sigma}
=
\frac{n^\sigma m^\sigma}{[n,m]^{2\sigma}}F\bigl(\tfrac{N}{[n,m]}\bigr).
\label{b4}
\end{align}
Using the asymptotics of $F$, 
it is easy to determine the entrywise asymptotic behaviour of these matrices:
$$
\frac1{F(N)}
[\bT_N^*(\psi_\sigma)\bT_N(\psi_\sigma)]_{nm}
=
\frac{n^\sigma m^\sigma}{[n,m]^{2\sigma}}\frac{F(\tfrac{N}{[n,m]})}{F(N)}
\to
\frac{n^\sigma m^\sigma}{[n,m]^{2\sigma}}[n,m]^{-\rho}
=
\frac{n^\sigma m^\sigma}{[n,m]}
$$
as $N\to\infty$, where we recognise $E(\sigma,1)$ in the right hand side. We need to prove that this convergence is also valid in the $\bS_q$ norm for any even integer $q$ with $q\rho>1$. 

Denote by $G_N$ the infinite matrix with entries
$$
[G_N]_{n,m}=[n,m]^\rho\frac{F(\frac{N}{[n,m]})}{F(N)}. 
$$
Formula \eqref{b4} can be written as 
$$
\frac1{F(N)}\bT_N(\psi_\sigma)^*\bT_N(\psi_\sigma)=E(\sigma,1)\odot G_N, 
$$
where $\odot$ is the Hadamard product (entrywise multiplication of matrices). 
In the rest of the proof we write $E$ instead of $E(\sigma,1)$ for readability.
By the properties of $F$,  we have a uniform estimate 
\begin{equation}
\abs{[G_N]_{n,m}}\leq C_\sigma,
\label{b5}
\end{equation}
where $C_\sigma$ is independent of $n,m,N$ and 
\begin{equation}
[G_N]_{n,m}\to1, \quad N\to\infty,
\label{b6}
\end{equation}
for any $n,m$. 
By Theorem~\ref{thm1}, we have $E\in\bS_q$ for any $q>1/\rho$. 
Since $q$ is an even integer, we have 
\begin{equation}
\norm{E}_{\bS_q}^q
=
\Tr(E^q)
=
\sum_{k_1,\dots,k_q} [E]_{k_1,k_2}\cdots [E]_{k_q,k_1} <\infty;
\label{b3}
\end{equation}
observe that all entries of $E$ are positive. 
Similarly,
\begin{align*}
\norm{E&\odot G_N-E}_{\bS_q}^q
=
\Tr\bigl((E\odot G_N-E)^q\bigr)
\\
&=
\sum_{k_1,\dots,k_q} [E]_{k_1,k_2}\cdots [E]_{k_q,k_1}
([G_N]_{k_1,k_2}-1)\cdots ([G_N]_{k_q,k_1}-1). 
\end{align*}
Now using \eqref{b3} together with the properties \eqref{b5}, \eqref{b6} of $G_N$ and applying the dominated convergence theorem, we see that the last sum goes to zero as $N\to\infty$. 
The proof of Theorem~\ref{thm2} is complete. \qed

\section{Spectral analysis of $E_{p}(\sigma,\tau)$}\label{sec.a}

\subsection{Preliminaries}
In this section we need some notation for infinite matrices, considered as operators in $\ell^2(\bbZ_+)$. If $\gamma=\{\gamma_k\}_{k=0}^\infty$ is a sequence of real numbers, we denote by $D(\gamma)$ the diagonal matrix with elements $\{\gamma_0,\gamma_1,\dots\}$ on the diagonal; in other words, $D(\gamma)$ is the operator of multiplication by the sequence $\gamma$ in $\ell^2(\bbZ_+)$. By a slight abuse of notation, we write $D(a^k)$ if $\gamma_k=a^k$. 
Let $S$ be the standard shift operator, defined by $(Sx)_k=x_{k-1}$ for $k\geq1$ and $(Sx)_0=0$. 
We denote the inner product in a Hilbert space by $\jap{\cdot,\cdot}$. 

As above, we denote by $E_p(\sigma,\tau)$ the infinite matrix $\{p^{\sigma j}p^{-\tau(j\vee k)}p^{\sigma k}\}_{j,k\in\bbZ_+}$, considered as a linear operator in $\ell^2(\bbZ_+)$. It will be convenient to consider $p>1$ as an arbitrary real number (not necessarily an integer prime). 

\subsection{Eigenvalue estimates for $E_p(\sigma,\tau)$}

\begin{lemma}
Let $p>1$. Then for any $a>\frac1{\sqrt{p}}\frac1{1-\frac1p}$ and any $x\in\ell^2(\bbZ_+)$ we have the inequalities
$$
\frac{1-\frac{a}{\sqrt{p}}}{1-\frac1p+\frac1{a\sqrt{p}}}
\jap{D(p^{-k})x,x}
\leq
\frac1{1-\frac1p}
\jap{E_p(0,1)x,x}
\leq 
\frac{1+\frac{a}{\sqrt{p}}}{1-\frac1p-\frac1{a\sqrt{p}}}
\jap{D(p^{-k})x,x}.
$$
\end{lemma}
\begin{remark*}
For the first inequality to be meaningful, we need $1-\frac{a}{\sqrt{p}}>0$, i.e. $a<\sqrt{p}$. This is compatible with the condition $a>\frac1{\sqrt{p}}\frac1{1-\frac1p}$ only for $p>2$. However, this is not of our concern here, because eventually we are interested in taking $p\to\infty$. 
\end{remark*}
\begin{proof}
\emph{Step 1:}
Our first step is to check the identity
\begin{equation}
\jap{E_p(0,1)x,x}
=
(1-p^{-1})\norm{D(p^{-\frac12 k})(I-S)^{-1}x}^2, \quad x\in\ell^2(\bbZ_+).
\label{c2}
\end{equation}
Here 
$$
(I-S)^{-1}x=(x_0,x_0+x_1,x_0+x_1+x_2,\dots)
$$
is a sequence which is of course not necessarily in $\ell^2(\bbZ_+)$, but the combination $D(p^{-\frac12 k})(I-S)^{-1}x$ is in $\ell^2(\bbZ_+)$, and so \eqref{c2} makes sense.

In order to verify \eqref{c2}, we recall that the matrix elements $\{p^{-j\vee k} \}$ of $E_p(0,1)$ are constant on the ``corners'' $j\vee k=\const$. Summing over these ``corners'', we obtain 
\begin{align*}
\jap{E_p(0,1)x,x}
=&
\abs{x_0}^2
+p^{-1}(\abs{x_0+x_1}^2-\abs{x_0}^2)
\\
&+p^{-2}(\abs{x_0+x_1+x_2}^2-\abs{x_0+x_1}^2)+\cdots
\\
=&
(1-p^{-1})\abs{x_0}^2
+(1-p^{-1})p^{-1}\abs{x_0+x_1}^2
\\
&+(1-p^{-1})p^{-2}\abs{x_0+x_1+x_2}^2+\cdots
\\
=&
(1-p^{-1})\sum_{k=0}^\infty p^{-k}\Abs{\sum_{j=0}^k x_j}^2
\\
=&
(1-p^{-1})\norm{D(p^{-\frac12 k})(I-S)^{-1}x}^2.
\end{align*}

\emph{Step 2:}
The second step is to verify the identity 
\begin{equation}
D(p^{-\frac12 k})(I-S)^{-1}x
=
D(p^{-\frac12 k})x
+
p^{-\frac12}SD(p^{-\frac12 k})(I-S)^{-1}x.
\label{c3}
\end{equation}
We start from the identity 
$$
D(p^{-\frac12 k})Sy=p^{-\frac12}SD(p^{-\frac12 k})y
$$
for any $y\in\ell^\infty(\bbZ_+)$. Rearranging and adding $D(p^{-\frac12 k})y$ to both sides, we obtain
$$
D(p^{-\frac12 k})y=D(p^{-\frac12 k})(I-S)y+p^{-\frac12}SD(p^{-\frac12 k})y. 
$$
Taking $y=(I-S)^{-1}x$ and observing that $(I-S)y=x$, we arrive at \eqref{c3}.

\emph{Step 3:}
Let us take the square of the norm of both sides of \eqref{c3} and use the isometricity of $S$:
\begin{align*}
\norm{D(p^{-\frac12 k})(I-S)^{-1}x}^2
=&
\norm{D(p^{-\frac12 k})x}^2
+
p^{-1}\norm{D(p^{-\frac12 k})(I-S)^{-1}x}^2
\\
&+2p^{-\frac12}\Re\jap{SD(p^{-\frac12 k})(I-S)^{-1}x,D(p^{-\frac12 k})x}.
\end{align*}
Using the inequality 
$$
2\abs{\Re\jap{x,y}}\leq a\norm{x}^2+\frac1a \norm{y}^2,
$$
from here we find 
\begin{align*}
\left(1-\tfrac1p-\tfrac1{a\sqrt{p}}\right)\norm{D(p^{-\frac12 k})(I-S)^{-1}x}^2
&\leq
\left(1+\tfrac{a}{\sqrt{p}}\right)\norm{D(p^{-\frac12 k})x}^2,
\\
\left(1-\tfrac1p+\tfrac1{a\sqrt{p}}\right)\norm{D(p^{-\frac12 k})(I-S)^{-1}x}^2
&\geq
\left(1-\tfrac{a}{\sqrt{p}}\right)\norm{D(p^{-\frac12 k})x}^2.
\end{align*}
Since $a>\frac1{\sqrt{p}}\frac1{1-\frac1p}$, both brackets on the left here are positive, so we can divide by these brackets.
Using \eqref{c2}, we arrive at the required estimate. 
\end{proof}

\begin{lemma}\label{lma.c1}
Let $\rho>0$, $\tau>0$ and $p>1$. Then $E_p(\sigma,\tau)$ is compact and positive definite and its eigenvalues 
$\{\lambda_k(E_p(\sigma,\tau))\}_{k=0}^\infty$ (listed in non-decreasing order) satisfy 
$$
\lambda_k(E_p(\sigma,\tau))=p^{-k\rho}(1+\calO(p^{-\frac12\tau})), \quad p\to\infty,
$$
with the constants in $\calO(p^{-\frac12\tau})$ independent on $k\in\bbZ_+$.
\end{lemma}
\begin{proof}
Observe that by the definition of $E_p(\sigma,\tau)$, we have
$$
E_p(\sigma,\tau)=D(p^{\sigma k})E_{p^\tau}(0,1)D(p^{\sigma k}).
$$
Now let us take $a=1/2$, say, in the previous lemma and apply it to $x=D(p^{\sigma k})y$ and to $E_{p^\tau}(0,1)$ in place of $E_{p}(0,1)$. We obtain, for sufficiently large $p$, 
$$
(1+\calO(p^{-\frac12\tau}))
\jap{D(p^{-\rho k})y,y}
\leq
\jap{E_p(\sigma,\tau)y,y}
\leq 
(1+\calO(p^{-\frac12\tau}))
\jap{D(p^{-\rho k})y,y}.
$$
Since the eigenvalues of $D(p^{-\rho k})$ are given by $\{p^{-\rho k}\}_{k=0}^\infty$, by the min-max principle we obtain the required upper and lower bounds for the eigenvalues of $E_p(\sigma,\tau)$. 
The positive definiteness of $E_p(\sigma,\tau)$ (for every $p$) also follows from the lower estimate of the previous lemma by choosing a suitable $a$. 
\end{proof}

\subsection{The top eigenvalue}

We denote by $e_0=(1,0,0,\dots)$ the first element of the standard basis in $\ell^2(\bbZ_+)$. 
\begin{lemma}\label{lma.c2}
Let $\rho>0$, $\tau+\rho>0$  and $p>1$. 
\begin{enumerate}[\rm (i)]
\item
For all sufficiently large $p$, the top eigenvalue of $E_p(\sigma,\tau)$ is simple and satisfies
$$
\lambda_0(E_p(\sigma,\tau))=1+\calO(p^{-(\tau+\rho)}), \quad p\to\infty.
$$
\item
Let $\varphi_{p}$ be the normalised eigenvector of $E_p(\sigma,\tau)$ corresponding to the top eigenvalue. 
Then 
$$
\abs{\jap{\varphi_{p},e_0}}=1+\calO(p^{-(\tau+\rho)}), \quad p\to\infty.
$$
\end{enumerate}
\end{lemma}

\noindent
{\bf Remark}\, Part (i) was proven in \cite{H} (for $\tau=1$) where it was shown that the eigenvalues are related to the zeros of some explicit entire function. Here we give a different proof which avoids that detailed analysis.  

\begin{proof}
Throughout the proof, we write $E_p$ instead of $E_p(\sigma,\tau)$ for readability.
The idea of the proof is to separate the first row and first column of $E_p$. 
We write 
$$\ell^2(\bbZ_+)=\bbC\oplus\ell^2(\bbN)$$ 
and with respect to this decomposition write $E_p$ as a $2\times2$ block matrix
$$
E_p=
\begin{pmatrix}
1 & \jap{\cdot,a_p}
\\
a_p & E_p^\perp
\end{pmatrix},
$$
where $a_p=\{p^{-k(\tau-\sigma)}\}_{k=1}^\infty\in \ell^2(\bbN)$, and $E_p^\perp$ is the matrix $E_p$ with the top row and first column removed. We first estimate the norms of $a_p$ and $E_p^\perp$. For $a_p$ we have 
$$
\norm{a_p}^2=\sum_{j=1}^\infty p^{-2k(\tau-\sigma)}=\frac{p^{-2(\tau-\sigma)}}{1-p^{-2(\tau-\sigma)}}=\calO(p^{-2(\tau-\sigma)})=\calO(p^{-(\tau+\rho)}),
$$
as $p\to\infty$. 
Before coming to $E_p^\perp$, we notice that under our assumptions, the Hilbert-Schmidt norm of $E_p$ is finite and bounded for $p\to\infty$:
\begin{align*}
\Tr (E_p)^2&=\sum_{j,k=0}^\infty p^{2\sigma(j+k)}p^{-2\tau(j\vee k)}
\leq
2\sum_{j=0}^\infty \sum_{k=j}^\infty p^{2\sigma(j+k)}p^{-2\tau k}
\\
&
=\frac2{1-p^{2(\sigma-\tau)}}\sum_{j=0}^\infty p^{(4\sigma-2\tau)j}
=\frac2{1-p^{2(\sigma-\tau)}}\sum_{j=0}^\infty p^{-2\rho j}=\frac2{1-p^{-2(\tau+\rho)}}\frac1{1-p^{-2\rho}}.
\end{align*}
It follows that 
$$
\Tr (E_p^\perp)^2=\sum_{j,k=1}^\infty p^{2\sigma(j+k)}p^{-2\tau(j\vee k)}
=p^{-2\rho}
\sum_{j,k=0}^\infty p^{2\sigma(j+k)}p^{-2\tau(j\vee k)}
=\calO(p^{-2\rho}),
$$
as $p\to\infty$. 

Let us prove (i). From the above norm estimates for $a_p$ and $E_p^\perp$ and from the fact that the Hilbert-Schmidt norm dominates the operator norm, we conclude that 
$$
\Norm{E_p-\begin{pmatrix} 1& 0\\ 0&0\end{pmatrix}}\to0, \quad p\to\infty
$$
and thus the top eigenvalue $\lambda_0(E_p)$ is simple (for sufficiently large $p$) and $\lambda_0(E_p)\to1$ as $p\to\infty$. 
We write an eigenvalue equation for $E_p$ on the eigenvector $x=(x_0,x_1)$, $x_0\in\bbC$, $x_1\in\ell^2(\bbN)$:
\begin{align*}
x_0+\jap{x_1,a_p}&=\lambda x_0,
\\
x_0a_p+E_p^\perp x_1&=\lambda x_1.
\end{align*}
Let $\lambda=\lambda_0(E_p)$; then for all sufficiently large $p$ the operator $\lambda I-E_p^\perp$ is invertible. Using this fact, we express $x_1$ from the second equation and substitute it in the first one. This yields
$$
\lambda=1+\jap{(\lambda-E_p^\perp)^{-1}a_p,a_p}, \quad \lambda=\lambda_0(E_p).
$$
Since the norm of the inverse $(\lambda_{0}(E_p)-E_p^\perp)^{-1}$ remains bounded as $p\to\infty$, we have 
$$
\lambda_0(E_p)=1+\calO(\norm{a_p}^2)=1+\calO(p^{-(\tau+\rho)}), \quad p\to\infty.
$$
Putting this together yields the required estimate. 

Let us prove (ii). By the spectral theorem for compact self-adjoint operators, the resolvent $(E_p-\lambda I)^{-1}$ has poles at the eigenvalues, with residues equal to (minus) the projection onto the corresponding eigenspace. It follows that 
$$
\abs{\jap{\varphi_{p},e_0}}^2=-\Res_{\lambda=\lambda_0(E_p)}\jap{(E_p-\lambda I)^{-1}e_0,e_0}.
$$
To compute the resolvent on the element $e_0$, we follow the same procedure as above: if $(E_p-\lambda I)^{-1}e_0=(x_0,x_1)$, then we have 
\begin{align*}
(1-\lambda)x_0+\jap{x_1,a_p}&=1,
\\
x_0a_p+(E_p^\perp-\lambda I)x_1&=0.
\end{align*}
Expressing $x_1$ from the second equation and substituting into the first one, we find
$$
\jap{(E_p-\lambda I)^{-1}e_0,e_0}=x_0=\frac{1}{1-\lambda+\jap{(\lambda-E_p^\perp)^{-1}a_p,a_p}}
$$
Computing the residue, we find
$$
\Res_{\lambda=\lambda_0(E_p)}\jap{(E_p-\lambda I)^{-1}e_0,e_0}
=
-
\frac1{1+\norm{(\lambda_0(E_p)-E_p^\perp)^{-1}a_p}^2}.
$$
Finally, just as in the proof of part (i), we find 
$$
\norm{(\lambda_0(E_p)-E_p^\perp)^{-1}a_p}^2=\calO(\norm{a_p}^2)=\calO(p^{-(\tau+\rho)}), \quad p\to\infty.
$$
Putting this together yields the required estimate. 
\end{proof}

\section{Spectral analysis of $E(\sigma,\tau)$}\label{sec.d}

\subsection{The product formula for the eigenvalues of $E(\sigma,\tau)$}
Our next aim is to prove that the eigenvalues of $E(\sigma,\tau)$ are given by products of the eigenvalues of $E_p(\sigma,\tau)$.
\begin{theorem}\label{thm.c3}
Let $\rho>0$ and $\tau+\rho>1$. Then the operator $E(\sigma,\tau)$ is bounded and has a pure point spectrum. 
There exists an enumeration $\{\lambda_n(E(\sigma,\tau))\}_{n=1}^\infty$ of the  eigenvalues of $E(\sigma,\tau)$ (counting multiplicities) such that
\begin{equation}
\lambda_n(E(\sigma,\tau))=\prod_{p \text{ \rm  prime}} \lambda_{k_p}(E_p(\sigma,\tau)), \quad\text{ if }\quad n=\prod_{p \text{ \rm prime}}p^{k_p}.
\label{b14}
\end{equation}
\end{theorem}
We observe that this theorem does not require $\tau>0$ (cf. \eqref{a0}). 
\begin{proof}[Proof of Theorem~\ref{thm.c3}]
We start by introducing some notation. We denote the standard basis in $\ell^2(\bbZ_+)$ by $e_j$, $j\in\bbZ_+$. Furthermore, writing each natural number $n$ as in \eqref{b14}, we can identify $\bbN$ with the subset $\bbZ_+^{(\infty)}$ of the infinite Cartesian product $\bbZ_+^\infty$; by definition, the subset $\bbZ_+^{(\infty)}$ consists of all sequences $\{k_p\}_{p \text{ prime}}$ such that $k_p=0$ for all but finitely many $p$'s.  In particular, we have $\ell^2(\bbN)=\ell^2(\bbZ_+^{(\infty)})$. The standard basis in $\ell^2(\bbZ_+^{(\infty)})$ will be denoted by $\bfe_j$, where $j\in \bbZ_+^{(\infty)}$. More precisely, 
$$
\bfe_j(k)=\delta_{j,k}=\prod_{p\text{ prime}} \delta_{j_p,k_p},\quad k\in \bbZ_+^{(\infty)}.
$$
In the rest of the proof, we write $E_p$ in place of $E_p(\sigma,\tau)$ for readability.

Heuristically, the strategy of the proof is to diagonalise $E(\sigma,\tau)$ by diagonalising $E_p$ in each term of the infinite tensor product \eqref{b3a}. However, we prefer to avoid the highly abstract language of infinite tensor products and express this idea in more concrete terms. 

For each $p$, let $D_p$ be the operator of multiplication by the (non-increasingly ordered) sequence of eigenvalues of $E_p$ in $\ell^2(\bbZ_+)$:
$$
D_pe_j=\lambda_j(E_p)e_j, \quad j\in\bbZ_+.
$$
In other words, $D_p$ is the infinite diagonal matrix with elements $\{\lambda_j(E_p)\}_{j=0}^\infty$ on the diagonal. 
Further, let $U_p$ be a unitary operator in $\ell^2(\bbZ_+)$ that diagonalises $E_p$, i.e. 
$$
U_p^*E_pU_p=D_p.
$$
The above relation allows for a multiplication of $U_p$ by a unimodular complex constant. We fix this constant so that the top left entry of the matrix representation of $U_p$ is positive:
$$
\jap{U_pe_0,e_0}>0.
$$
Observe that $U_pe_0$ is the eigenvector of $E_p$ corresponding to the top eigenvalue $\lambda_0(E_p)$. Recalling Lemma~\ref{lma.c2}(ii), we find that with this normalisation,
$$
\jap{U_pe_0,e_0}=1+\calO(p^{-(\tau+\rho)}), \quad p\to\infty, 
$$
and therefore the infinite product 
\begin{equation}
\prod_{p\text{ prime}} \jap{U_pe_0,e_0}
\label{b11b}
\end{equation}
converges. 

Next, we define the operator $U$ in $\ell^2(\bbN)=\ell^2(\bbZ_+^{(\infty)})$ through its matrix representation in the standard basis as follows:
$$
\jap{U\bfe_j,\bfe_k}=\prod_{p\text{ prime}}\jap{U_pe_{j_p},e_{k_p}}.
$$
Since $j_p=k_p=0$ for all but finitely many primes $p$, we see that the convergence of this infinite product follows from the convergence of the product \eqref{b11b}. 

From this definition it is not even obvious that $U$ is bounded. However, we shall prove that in fact $U$ is unitary. 
First we need to introduce some notation, related to representing infinite sums over $\bbZ_+^{(\infty)}$ as limits of finite sums. 
For each $p$, let us fix an integer $L_p\geq0$ such that $L_p=0$ for all sufficiently large $p$. Now 
let $\calL$ be a set of the following form:
$$
\calL=\{\ell\in\bbZ_+^{(\infty)}: 0\leq \ell_p\leq L_p\}.
$$
If $\calL^{(i)}$ is a sequence of sets of this form, we will write $\calL^{(i)}\to\bbZ_+^{(\infty)}$, if $L_p^{(i)}\to\infty$ as $i\to\infty$ for all $p$. In particular, for such sequence we have $\cup_i \calL^{(i)}=\bbZ_+^{(\infty)}$. Below we suppress the dependence on the index $i$. 

For a given $j\in\bbZ_+^{(\infty)}$, let us compute the norm of $U\bfe_j$: 
$$
\norm{U\bfe_j}^2
=
\lim_{\calL\to\bbZ_+^{(\infty)}}
\sum_{k\in \calL}\Abs{\prod_{p\text{ prime}}\jap{U_pe_{j_p},e_{k_p}}}^2
$$
and the norm is finite as long as the limit exists. We have
$$
\sum_{k\in \calL}\Abs{\prod_{p\text{ prime}}\jap{U_pe_{j_p},e_{k_p}}}^2
=\sum_{k\in\calL}\prod_{p\text{ prime}}\abs{\jap{U_pe_{j_p},e_{k_p}}}^2
=\prod_{p\text{ prime}}\sum_{k_p=0}^{L_p}\abs{\jap{U_pe_{j_p},e_{k_p}}}^2
$$
(to understand why the second equality is true, expand the product in the right hand side).
But 
$$
\lim_{L_p\to\infty}\sum_{k_p=0}^{L_p}\abs{\jap{U_pe_{j_p},e_{k_p}}}^2=1
$$
for each prime $p$, because each $U_p$ is unitary. 
It follows that 
$$
\lim_{\calL\to\bbZ_+^{(\infty)}}
\prod_{p\text{ prime}}\sum_{k_p=0}^{L_p}\abs{\jap{U_pe_{j_p},e_{k_p}}}^2
=1
$$
and so $\norm{U\bfe_j}^2=1$. 

Our next aim is to prove that 
\begin{equation}
\jap{U\bfe_j,U\bfe_k}=\delta_{j,k},\quad j,k\in\bbZ_+^{(\infty)}.
\label{b11c}
\end{equation}
We have
$$
\jap{U\bfe_j,U\bfe_k}=\sum_{\ell\in\bbZ_+^{(\infty)}}\jap{U\bfe_j,\bfe_\ell}\jap{\bfe_\ell,U\bfe_k},
$$
where the series converges absolutely by Cauchy-Schwarz and the previous step of the proof. 
Using the definition of $U$ and the unitarity of $U_p$, we find
\begin{align*}
\sum_{\ell\in \bbZ_+^{(\infty)}}\jap{U\bfe_j,\bfe_\ell}\jap{\bfe_\ell,U\bfe_k}
&=
\sum_{\ell\in \bbZ_+^{(\infty)}}\prod_{p\text{ prime}}\jap{U_pe_{j_p},e_{\ell_p}}\jap{e_{\ell_p},U_pe_{k_p}}
\\
&=
\prod_{p\text{ prime}}
\sum_{\ell_p=0}^{\infty} \jap{U_pe_{j_p},e_{\ell_p}}\jap{e_{\ell_p},U_pe_{k_p}}
\\
&=
\prod_{p\text{ prime}}\jap{U_pe_{j_p},U_pe_{k_p}}
=
\prod_{p\text{ prime}} \delta_{j_p,k_p}=\delta_{j,k},
\end{align*}
where the interchange of limiting operations can be justified as on the previous step, by passing to the finite sums and using the absolute convergence. 
We have proved \eqref{b11c} and therefore $U$ is unitary. 

Consider the operator $D$ in $\ell^2(\bbZ_+^{(\infty)})$, defined by its action on the standard basis $\{\bfe_{k}\}_{k\in\bbZ_+^{(\infty)}}$ by 
$$
D\bfe_k= \left(\prod_{p \text{ prime}} \lambda_{k_p}(E_p)\right)\bfe_k.
$$
In other words, $D$ is the operator of multiplication by the sequence \eqref{b14} in $\ell^2(\bbZ_+^{(\infty)})$. 
We first note that by Lemma~\ref{lma.c2}(i), the infinite product above converges and is uniformly bounded; thus, $D$ is a bounded operator. Our aim is to check that 
$$
E(\sigma,\tau)=UDU^*.
$$
In other words, we need to check that the matrix representation of the operator $UDU^*$ in the standard basis in $\ell^2(\bbZ_+^{(\infty)})$ is given by the matrix $E(\sigma,\tau)$. For given $j,k\in\bbZ_+^{(\infty)}$ we have
$$
\jap{UDU^*\bfe_j,\bfe_k}
=
\jap{DU^*\bfe_j,U^*\bfe_k}
=
\sum_{\ell\in \bbZ_+^{(\infty)}}
\left(\prod_{p \text{ prime}}\lambda_{\ell_p}(E_p)\right)\jap{U^*\bfe_j,\bfe_\ell}\jap{\bfe_\ell,U^*\bfe_k},
$$
where the sum is absolutely convergent by Cauchy-Schwarz because $U^*\bfe_j$ and $U^*\bfe_k$ are in $\ell^2(\bbZ_+^{(\infty)})$ and the eigenvalues of $D$ are uniformly bounded. Recalling the definition of $U$ and interchanging the summation and the infinite product as before, we find

\begin{align*}
\jap{UDU^*\bfe_j,\bfe_k}
&=
\sum_{\ell\in\bbZ_+^{(\infty)} }\prod_{p \text{ prime}} \lambda_{\ell_p}(E_p)\jap{U_p^*e_{j_p},e_{\ell_p}}\jap{e_{\ell_p},U_p^*e_{k_p}}
\\
&=
\prod_{p \text{ prime}}
\sum_{\ell_p=0}^{\infty}\lambda_{\ell_p}(E_p)\jap{U_p^*e_{j_p},e_{\ell_p}}\jap{e_{\ell_p},U_p^*e_{k_p}}.
\end{align*}
By the definition of $U_p$, we have $U_pD_pU_p^*=E_p$, and so 
$$
\sum_{\ell_p=0}^{\infty}\lambda_{\ell_p}(E_p)\jap{U_p^*e_{j_p},e_{\ell_p}}\jap{e_{\ell_p},U_p^*e_{k_p}}
=
\frac{p^{\sigma j_p}p^{\sigma k_p}}{p^{\tau(j_p\vee k_p)}}.
$$
It follows that 
$$
\jap{UDU^*\bfe_j,\bfe_k}=\prod_{p\text{ prime}} \frac{p^{\sigma j_p}p^{\sigma k_p}}{p^{\tau(j_p\vee k_p)}}.
$$
Recalling \eqref{b300}, we conclude that $E(\sigma,\tau)=UDU^*$ as claimed. It follows that the spectrum of $E(\sigma,\tau)$ is pure point and is given by the eigenvalues \eqref{b14}.
\end{proof}

\subsection{Proof of Theorem~\ref{thm1}}
Our main task is to establish the asymptotics \eqref{s9}. 
First let us introduce the eigenvalue counting function for $E(\sigma,\tau)$ as follows:
$$
\mu(t)=\#\{n\in\bbN: \lambda_n(E(\sigma,\tau))>1/t\}, \quad t>0.
$$
It is clear that $\mu$ is non-decreasing and $\mu(t)=0$ for all sufficiently small $t$ (more precisely, for $t<\norm{E(\sigma,\tau)}^{-1}$). Furthermore, it is easy to see that the desired eigenvalue asymptotics \eqref{s9} is equivalent to the asymptotics 
\begin{equation}
\mu(t)=\varkappa(\sigma,\tau)^{-1/\rho}t^{1/\rho}+o(t^{1/\rho}), \quad t\to\infty. 
\label{b12}
\end{equation}
Below our aim is to prove \eqref{b12}. 
We do this by using the following Tauberian theorem due to Wiener and Ikehara, see e.g. \cite[Theorem 4.1]{Korevaar} (the version given below differs only by scaling). 
\begin{theorem}\label{thm.tauberian}
Let $\mu=\mu(t)$ be a non-negative non-decreasing function on $\bbR$, vanishing for all $t\leq t_0$ with some $t_0>0$ and such that the Mellin-Stieltjes transform 
$$
f(s)=\int_{t_0}^\infty t^{-s}d\mu(t)
$$
exists (i.e. the integral converges absolutely) for $\Re s>s_0>0$. Suppose that for some $A>0$, the analytic function 
$$
f(s)-\frac{A}{s-s_0}, \quad \Re s>s_0
$$
has a continuous extension to the closed half-plane $\Re s\geq s_0$. Then 
$$
\mu(t)t^{-s_0}\to A/s_0, \quad t\to\infty.
$$
\end{theorem}

\begin{proof}[Proof of Theorem~\ref{thm1}]
By Lemma~\ref{lma.c1}, all eigenvalues of $E_p$ are positive. By the product formula \eqref{b14}, it follows that all eigenvalues of $E(\sigma,\tau)$ are also positive, and so $E(\sigma,\tau)$ is positive definite with trivial kernel. The compactness will follow from the considerations below. 

Let us prove the asymptotic formula \eqref{b12}.
We will use Theorem~\ref{thm.tauberian} with $s_0=1/\rho$. Consider the function 
\begin{equation}
f(s)=\int_{t_0}^\infty t^{-s}d\mu(t)=\sum_{n=1}^\infty \lambda_n(E(\sigma,\tau))^s, \quad \Re s>1/\rho.
\label{b18}
\end{equation}
Our aim is to show that the integral above converges absolutely in the open half-plane $\Re s>1/\rho$ and possesses the analytic properties indicated in Theorem~\ref{thm.tauberian}. By the factorisation \eqref{b14}, we have 
\begin{equation}
f(s)=\prod_{p\text{ prime}}f_p(s), 
\quad
f_p(s)=\sum_{k=0}^\infty \lambda_k(E_p)^s. 
\label{b17}
\end{equation}
By Lemma~\ref{lma.c1}, the eigenvalues of $E_p$ decay exponentially, and therefore the series in the right hand side here converges absolutely in the half-plane $\Re s>0$ and defines an analytic function in this half-plane. Let us estimate $f_p(s)$; we use Lemma~\ref{lma.c2} for $\lambda_0(E_p)$ and Lemma~\ref{lma.c1} for $\lambda_k(E_p)$ with $k\geq1$: 
\begin{align*}
f_p(s)&=(1+\calO(p^{-(\tau+\rho)}))^s+\sum_{k=1}^\infty(1+\calO(p^{-\frac12\tau}))^sp^{-\rho ks}
\\
&=1+\calO_s(p^{-(\tau+\rho)})+\frac{p^{-\rho s}}{1-p^{-\rho s}}+\calO_s(p^{-\frac12\tau -\rho s})
\\
&=1+p^{-\rho s}+\calO_s(p^{-(\tau+\rho)})+\calO_s(p^{-2\rho s})+\calO_s(p^{-\frac12\tau-\rho s}),
\end{align*}
where $\calO_s$ means that the constants in the estimates are uniform in $s$ over compact sets in the half-plane $\Re s>0$. 
This estimate already shows that the infinite product \eqref{b17} converges absolutely for $\Re s>1/\rho$ and therefore the Mellin-Stieltjes transform in \eqref{b18} also converges absolutely. 
Now rewrite this as
\begin{align*}
f_p(s)&=\frac1{1-p^{-\rho s}}g_p(s), 
\\
g_p(s)&=1+\calO_s(p^{-(\tau+\rho)})+\calO_s(p^{-2\rho s})+\calO_s(p^{-\frac12\tau-\rho s}).
\end{align*}
It follows that for $\Re s>1/\rho$ we have 
$$
f(s)=\prod_{p\text{ prime}}(1-p^{-\rho s})^{-1} \prod_{p\text{ prime}} g_p(s)=\zeta(\rho s)g(s), \quad 
g(s)=\prod_{p\text{ prime}} g_p(s).
$$
By the above estimates for $g_p$, the infinite product for $g(s)$ converges in the half-plane where $2\rho \Re s>1$ and $\frac12\tau+\rho\Re s>1$, i.e. in the half-plane $\Re s>s_0$, where 
$$
s_0=\max\{\tfrac1{2\rho},\tfrac{2-\tau}{2\rho}\}<\tfrac1\rho.
$$ 
Thus, $g(s)$ is analytic in this half-plane. 
Since $\zeta(s)$ has a pole at $s=1$ with residue $=1$,  we find that 
$$
f(s)-\frac{g(1/\rho)/\rho}{s-1/\rho}
$$
is analytic in the half-plane $\Re s>s_0$.
Applying the Wiener-Ikehara theorem, we find that $\mu(t)t^{-\frac1\rho}\to g(1/\rho)$ as $t\to\infty$. 
This proves \eqref{b12} and therefore \eqref{s9} with 
\begin{equation}
\varkappa(\sigma,\tau)=g(1/\rho)^{-\rho}.
\label{b19}
\end{equation}
Note that 
\begin{equation}
g(1/\rho) = \prod_{p\text{ prime}} g_p(1/\rho) = \prod_{p\text{ prime}}\Bigl(1-\frac{1}{p}\Bigr)\sum_{k=0}^\infty \lambda_k(E_p)^{1/\rho} > 0,
\label{b16}
\end{equation}
as each term in the product is positive.
\end{proof}

\subsection{Computation of $\varkappa(\sigma,\tau)$ in some cases} \label{sec.d3}
We compute the value of $\varkappa(\sigma,\tau)$ corresponding to the cases $\rho=1$ and $\rho=1/2$. 

For $\rho=1$ we have $\tau=1+2\sigma$ and \eqref{b16} yields
$$
\varkappa(\sigma,1+2\sigma)
=\prod_{p\text{ prime}}\Bigl(1-\frac{1}{p}\Bigr)\sum_{k=0}^\infty \lambda_k(E_p(\sigma,1+2\sigma))
=\prod_{p\text{ prime}}\Bigl(1-\frac{1}{p}\Bigr)\Tr E_p(\sigma,1+2\sigma). 
$$
Recalling that the matrix elements of $E_p(\sigma,\tau)$ are $\{p^{\sigma k}p^{-\tau(j\vee k)}p^{\sigma j}\}_{j,k=0}^\infty$, we compute the trace as the sum of the elements on the diagonal:
$$
\Tr E_p(\sigma,1+2\sigma)=\sum_{k=0}^\infty p^{-\rho k}=\sum_{k=0}^\infty p^{-k}=\frac1{1-\frac1p},
$$
and therefore $\varkappa(\sigma,1+2\sigma)=1$.

If $\rho=1/2$, we have $\tau=\frac12+2\sigma$, and \eqref{b16} yields
$$
g(2)
=\prod_{p\text{ prime}}\Bigl(1-\frac{1}{p}\Bigr)\sum_{k=0}^\infty \lambda_k\bigl(E_p(\sigma,\tfrac12+2\sigma)\bigr)^2
=\prod_{p\text{ prime}}\Bigl(1-\frac{1}{p}\Bigr)\Tr (E_p(\sigma,\tfrac12+2\sigma))^2.
$$
A direct computation gives 
\begin{align*}
\Tr (E_p(\sigma,\tfrac12+2\sigma))^2&=\sum_{j,k=0}^\infty p^{2\sigma(j+k)}p^{-(1+4\sigma)(j\vee k)}
\\
&=\sum_{j=0}^\infty p^{2\sigma j}
\left(\sum_{k=0}^j p^{2\sigma k}p^{-(1+4\sigma)j}+\sum_{k=j+1}^\infty p^{2\sigma k}p^{-(1+4\sigma)k}\right)
\\
&=\sum_{j=0}^\infty p^{2\sigma j}
\left(p^{-(1+4\sigma)j}\frac{p^{2\sigma(j+1)}-1}{p^{2\sigma}-1}+p^{-(1+2\sigma) (j+1)}\frac1{1-p^{-(1+2\sigma)}}\right)
\\
&=
\frac1{p^{2\sigma}-1}\left(p^{2\sigma}\sum_{j=0}^\infty p^{-j}-\sum_{j=0}^\infty p^{-(1+2\sigma) j}\right)
+\frac{p^{-(1+2\sigma)}}{1-p^{-(1+2\sigma)}}\sum_{j=0}^\infty p^{-j}
\\
&=
\frac1{p^{2\sigma}-1}\left(\frac{p^{2\sigma}}{1-p^{-1}}-\frac1{1-p^{-(1+2\sigma)}}\right)+\frac{p^{-(1+2\sigma)}}{1-p^{-(1+2\sigma)}}\frac1{1-p^{-1}}.
\end{align*}
After some elementary algebra, we obtain
$$
(1-p^{-1})\Tr (E_p(\sigma,\tfrac12+2\sigma))^2=\frac{1+p^{-(1+2\sigma)}}{1-p^{-(1+2\sigma)}}=\frac{1-p^{-(2+4\sigma)}}{(1-p^{-(1+2\sigma)})^2}, 
$$
and therefore
$$
g(2)=\zeta(1+2\sigma)^2/\zeta(2+4\sigma), \quad \varkappa(\sigma,\tfrac12+2\sigma)=g(2)^{-1/2}=\sqrt{\zeta(2+4\sigma)}/\zeta(1+2\sigma).
$$

\subsection{Connection with generalised prime systems}\label{sec.d4}
The purpose of this subsection is two-fold. 
Firstly we point out a connection of this topic with the subject of generalised prime systems. 
Secondly, using this connection, we indicate (without going into technical details) how to give a sharper error bound in formula \eqref{s9}:
$$
\lambda_n(E(\sigma,\tau))=\frac{\varkappa(\sigma,\tau)}{n^\rho}+O(n^{-\rho}e^{-(\sqrt{\tau}-\eps)\sqrt{\log n\log\log n}}), \quad n\to\infty,
$$
for any $\eps>0$.

Briefly, a generalized prime system $\calP$ consists of a sequence of real numbers $1<r_1\leq r_2\leq \cdots \leq r_n \to\infty$ (called Beurling primes) and all possible products of powers of these (called Beurling integers), say, $1=n_1<n_2 \leq \cdots$.
One associates with $\calP$ the Beurling zeta function 
$$
\zeta_{\calP}(s)=\prod_{k=1}^\infty \frac1{1-r_k^{-s}}=\sum_{k=1}^\infty \frac1{n_k^s}.
$$
See e.g. \cite{DZ} for the details. 

Now let us denote 
$$
\gamma_n=\frac{\lambda_n(E(\sigma,\tau))}{\lambda_1(E(\sigma,\tau))}, \quad n\in\bbN,
\quad\text{ and }\quad
\gamma_{k,p}=\frac{\lambda_k(E_p(\sigma,\tau))}{\lambda_0(E_p(\sigma,\tau))}, \quad k\in\bbZ_+
$$
and consider $\{\gamma_{1,p}^{-1/\rho}\}_{p\text{ prime}}$ as a sequence of Beurling primes. 
By Lemmas~\ref{lma.c1} (with $k=1$) and \ref{lma.c2}(i), this sequence satisfies
$$
\gamma_{1,p}^{-1/\rho}=p+\calO(p^{1-\frac12\tau}), \quad p\to\infty.
$$
This case was discussed by Balazard \cite[Theorem~9]{Balazard} where it is proven that the counting function of the corresponding Beurling integers satisfies
$$
\#\{k: n_k<x\}=cx(1+O(e^{-(\sqrt{\tau}-\eps)\sqrt{\log x\log\log x}})), \quad x\to\infty,
$$
for some $c>0$ and any $\eps>0$. Rescaling, we find that the counting function of the sequence 
\begin{equation}
\widetilde \lambda_n=\prod_{p\text{ prime}}\gamma_{1,p}^{k_p}, 
\quad n=\prod_{p\text{ prime}}p^{k_p}
\label{b20}
\end{equation}
satisfies 
$$
\#\{n: \widetilde \lambda_n>\lambda\}=c\lambda^{-1/\rho}(1+O(e^{-(\sqrt{\tau/\rho}-\eps)\sqrt{\abs{\log \lambda}\log\abs{\log \lambda}}})), \quad \lambda\to0_+,
$$
for any $\eps>0$. 

Now it remains to show that the counting function of the eigenvalues of $E(\sigma,\tau)$ obeys the same asymptotic relation. 
In other words, we claim that replacing $\gamma_{1,p}^{k_p}$ by $\gamma_{k_p,p}$ in \eqref{b20} does not affect the asymptotics, i.e. the eigenvalues $\lambda_{k}(E_p(\sigma,\tau))$ with $k\geq2$ act somewhat similar to higher powers of primes. 

To see that, we come back to the Mellin-Stieltjes transform \eqref{b18} and find
\begin{align*}
f(s)&=\sum_{n=1}^\infty \lambda_n(E(\sigma,\tau))^s=\lambda_1(E(\sigma,\tau))^s\sum_{n=1}^\infty \gamma_n^s
=\lambda_1(E(\sigma,\tau))^s\prod_{p\text{ prime}}\left(1+\sum_{k=1}^\infty \gamma_{k,p}^s\right)
\\
&=\lambda_1(E(\sigma,\tau))^s\left(\prod_{p\text{ prime}}\frac1{1-\gamma_{1,p}^s}\right)\prod_{p\text{ prime}}(1-\gamma_{1,p}^s)\sum_{k=0}^\infty \gamma_{k,p}^s=\zeta_{\calP}(\rho s)\widetilde g(s),
\end{align*}
where $\zeta_{\calP}$ is the Beurling zeta function of our generalised prime system and $\widetilde g$ is a function which is (by a calculation similar to the one in the proof of Theorem~\ref{thm1}) analytic and bounded in an open half-plane containing $\Re s\geq1/\rho$. 
From here it is not difficult to show that the main contribution to the eigenvalue counting function is given by the term $\zeta_{\calP}(\rho s)$, which has a simple pole at $s=1/\rho$. 

\end{document}